\RequirePackage{fix-cm}
\documentclass[11pt,oneside,latin9,utf8]{amsart}
\usepackage[T1]{fontenc}
\usepackage[latin9]{inputenc}
\usepackage{color}
\usepackage{textcomp}
\usepackage{amsthm}
\usepackage{amstext}
\usepackage{amssymb}
\usepackage{fixltx2e}
\usepackage{setspace}
\usepackage{esint}
\usepackage[unicode=true,pdfusetitle,
 bookmarks=true,bookmarksnumbered=true,bookmarksopen=true,bookmarksopenlevel=1,
 breaklinks=true,pdfborder={0 0 1},backref=false,colorlinks=true]
 {hyperref}
\hypersetup{
 linkcolor=black, citecolor=black, urlcolor=blue, filecolor=blue,pdfpagelayout=OneColumn, pdfnewwindow=true, pdfstartview=XYZ, plainpages=true}
\usepackage{breakurl}

\makeatletter
\usepackage{enumitem}		% customizable list environments
\theoremstyle{plain}
\newtheorem{thm}{\protect\theoremname}
  \theoremstyle{definition}
  \newtheorem{defn}[thm]{\protect\definitionname}
  \theoremstyle{plain}
  \newtheorem{conjecture}[thm]{\protect\conjecturename}
  \theoremstyle{plain}
  \newtheorem{prop}[thm]{\protect\propositionname}
  \theoremstyle{plain}
  \newtheorem{cor}[thm]{\protect\corollaryname}
  \theoremstyle{remark}
  \newtheorem*{acknowledgement*}{\protect\acknowledgementname}
  \theoremstyle{remark}
  \newtheorem*{notation*}{\protect\notationname}
  \theoremstyle{plain}
  \newtheorem{lem}[thm]{\protect\lemmaname}

\usepackage{ifpdf} % part of the hyperref bundle
\ifpdf % if pdflatex is used

\usepackage{tikz}

\usepackage{lmodern}

\fi % end if pdflatex is used

\author{Michel van Garrel}
\address{KIAS, 85 Hoegiro, Dongdaemun-gu, Seoul 130-722, Republic of Korea}
\curraddr{}
\email{vangarrel@kias.re.kr}
\thanks{}

\author{Tony W. H. Wong}
\address{Department of Mathematics, Kutztown University of Pennsylvania, 15200 Kutztown Road, Kutztown, PA 19530, USA}
\curraddr{}
\email{wong@kutztown.edu}
\thanks{}

\author{Gjergji Zaimi}
\address{Department of Mathematics, California Institute of Technology, MC 253-37, 1200 East California Boulevard, Pasadena, CA 91125, USA}
\curraddr{}
\email{gzaimi@caltech.edu}
\thanks{}

\let\myTOC\tableofcontents
\renewcommand\tableofcontents{%
 \pdfbookmark[1]{\contentsname}{}
  \myTOC
}

 \def\LyX{\texorpdfstring{%
  L\kern-.1667em\lower.25em\hbox{Y}\kern-.125emX\@}
  {LyX}}

\newcommand{\Q}{\mathbb Q}
\newcommand{\Z}{\mathbb Z}
\newcommand{\N}{\mathbb N}

\newcommand{\IP}{\mathbb P}
\newcommand{\ptwo}{\IP^2}

\newcommand{\IC}{\mathbb C}

\newcommand{\mc}{\mathcal}

\DeclareMathOperator{\hhh}{H}
\DeclareMathOperator{\euler}{e}
\DeclareMathOperator{\tott}{Tot}
\DeclareMathOperator{\dt}{DT}
\DeclareMathOperator{\hilb}{Hilb}

\newcommand{\xyR}[1]{
  \xydef@\xymatrixrowsep@{#1}}
\newcommand{\xyC}[1]{
  \xydef@\xymatrixcolsep@{#1}}

\makeatother

  \providecommand{\acknowledgementname}{Acknowledgement}
  \providecommand{\conjecturename}{Conjecture}
  \providecommand{\corollaryname}{Corollary}
  \providecommand{\definitionname}{Definition}
  \providecommand{\lemmaname}{Lemma}
  \providecommand{\notationname}{Notation}
  \providecommand{\propositionname}{Proposition}
\providecommand{\theoremname}{Theorem}

\makeatletter
\def\blfootnote{\xdef\@thefnmark{}\@footnotetext}
\makeatother

\begin{document}

\title{Integrality of relative BPS state counts of toric Del Pezzo surfaces}

\begin{abstract}
Relative BPS state counts for log Calabi-Yau surface pairs were introduced
by Gross-Pandharipande-Siebert in \cite{Gross-Panda-Siebert} and
conjectured by the authors to be integers. For toric Del Pezzo surfaces,
we provide an arithmetic proof of this conjecture, by relating these
invariants to the local BPS state counts of the surfaces. The latter
were shown to be integers by Peng in \cite{GV_inv_integral_local_toric_CY};
and more generally for toric Calabi-Yau threefolds by Konishi in \cite{Konishi_int_GV_inv}.
\end{abstract}

\maketitle

\blfootnote{M. van Garrel was supported by a Fields Postdoctoral Fellowship for the Fields major thematic program on Calabi-Yau Varieties: Arithmetic, Geometry and Physics from July to December 2013.}

\pagestyle{myheadings}
\markright{\hfill Integrality of relative BPS state counts of toric Del Pezzo surfaces \hfill}

\section{Introduction}

\subsection{Local BPS state counts}\label{intro-loc-BPS}

For Calabi-Yau threefolds, \emph{BPS invariants} were defined by Gopakumar-Vafa
in \cite{GV_M_theory_I,GV_M_theory_II} using a $M$-theory construction.
Their definition and the related conjectures were extended to all
threefolds by Pandharipande in \cite{panda_Hodge_int_deg_contr,panda_3_qu_GW}.
Presently, genus $0$ invariants are considered, also called \emph{BPS
state counts}. For the local Calabi-Yau geometries relevant below,
the terminology \emph{local} BPS state counts is used. Let $S$ be
a smooth Del Pezzo surface and denote by $E$ a smooth effective anticanonical
divisor on it, that is, an elliptic curve. Denote furthermore by $K_{S}$
the non-compact local Calabi-Yau threefold given as the total space
of the canonical bundle $\mc O_{S}(-E)$ on $S$. For a curve class
$\beta\in\hhh_{2}(S,\Z)$, denote by $n_{\beta}$ the local BPS state
count in class $\beta$, whose definition we state below. From
a physics point of view, $n_{\beta}$ counts $D$-branes supported
on genus 0 curves of class $\beta$. This definition does not rest
on rigorous mathematical foundations, so alternative definitions are
used. Perhaps the most common one, the one we follow in this paper,
aims at extracting multiple cover contributions from Gromov-Witten
invariants and is as follows. Denote by $\overline{M}_{0,0}(S,\beta)$,
resp. by $\overline{M}_{0,1}(S,\beta)$, the moduli stack of stable
maps $f:C\to S$ from genus $0$ curves with no, resp. one, marked
point to $S$ such that $f_{*}([C])=\beta$. Denote moreover by $\pi:\overline{M}_{0,1}(S,\beta)\to\overline{M}_{0,0}(S,\beta)$
the forgetful morphism and by $ev:\overline{M}_{0,1}(S,\beta)\to S$
the evaluation map. This determines the obstruction bundle $R^{1}\pi_{*}ev^{*}K_{S}$
whose fiber over a stable map $f:C\to S$ is $\hhh^{1}(C,f^{*}K_{S})$.
Then \emph{$I_{K_{S}}(\beta)$,} \emph{the genus $0$ local Gromov-Witten
invariant of degree $\beta$ }of\emph{ $S$,} is defined as the integral
of the Euler class of $R^{1}\pi_{*}ev^{*}K_{S}$ against the virtual
fundamental class of $\overline{M}_{0,0}(S,\beta)$, i.e.
\[
I_{K_{S}}(\beta):=\int_{[\overline{M}_{0,0}(S,\beta)]^{vir}}\euler\left(R^{1}\pi_{*}ev^{*}K_{S}\right)\in\mathbb{Q}.
\]
The definition of the associated BPS state counts is modeled on the
following ideal (rarely satisfied) situation: Suppose that $K_{S}$
only contained a finite number of genus 0 degree $\beta$ curves, and
that all these curves were rigidly embedded in $K_{S}$, i.e.\ with
normal bundle isomorphic to $\mc O(-1)\oplus\mc O(-1)$. 
Then $n_{\beta}$ should be the number of such curves. Let $\tilde{C}\subset K_{S}$
be such a rigid rational curve of degree $\beta$. According to the
Aspinwall-Morrison formula proven by Manin in \cite{manin_am_et_al},
degree $k$ stable maps 
\[
C\to\tott(\mc O_{\tilde{C}}(-1)\oplus\mc O_{\tilde{C}}(-1))
\]
contribute a factor of $\frac{1}{k^{3}}$ to the Gromov-Witten invariant
$I_{K_{S}}(k\beta)$. For $k\in\N$, we write $k|\beta$ to mean that
there is $\beta'\in\hhh_{2}(K_{S},\Z)$ such that $k\beta'=\beta$.
In this ideal situation then, the following equality would hold:
\begin{equation}
I_{K_{S}}(\beta)=\sum_{k|\beta}\frac{1}{k^{3}}n_{\beta/k}.\label{eq:def_gv_form}
\end{equation}
In general, the described geometric conditions are not satisfied,
and so the $n_{\beta}$ do not count curves in class $\beta$.%
\footnote{It is believed though that a (non-algebraic) deformation of $K_{S}$ exhibits
these conditions.%
} They can nonetheless be defined via equation \eqref{eq:def_gv_form},
or alternatively, via generating functions as follows.
\begin{defn}
(Stated as a formula by Gopakumar-Vafa in \cite{GV_M_theory_I,GV_M_theory_II};
stated as a definition by Bryan-Pandharipande in \cite{bryan_panda_bps}).
Assume $\beta$ to be primitive. Then the local BPS state counts $n_{d\beta}$,
for $d\geq1$, are defined as rational numbers via the formula
\begin{equation}
\sum_{l=1}^{\infty}I_{K_{S}}(l\beta)\, q^{l}=\sum_{d=1}^{\infty}n_{d\beta}\sum_{k=1}^{\infty}\frac{1}{k^{3}}\, q^{dk}.\label{eq:defn_gv_gen_fn}
\end{equation}
\end{defn}
\begin{conjecture}
\label{conjecture_GV}(Attributed to Gopakumar-Vafa; stated in \cite{bryan_panda_bps} by Bryan-Pandharipande). For all curve classes $\beta\in\hhh_{2}(S,\Z)$,
\[
n_{\beta}\in\Z.
\]
\end{conjecture}
Conjecture \ref{conjecture_GV} was proven by Peng in \cite{GV_inv_integral_local_toric_CY}
in the case of toric Del Pezzo surfaces, which are the Del Pezzo surfaces
of degree $\geq6$. More generally, a proof for toric Calabi-Yau threefolds
was given by Konishi in \cite{Konishi_int_GV_inv}.

\subsection{Relative BPS state counts}\label{intro-rel-BPS}

The definitions and conjectures relating to relative BPS state counts
mirror the discussion of the previous section. These invariants were
introduced by Gross-Pandharipande-Siebert in \cite{Gross-Panda-Siebert}.
Whereas the previous section is concerned with \emph{local Calabi-Yau threefolds},
this one deals with \emph{open Calabi-Yau surfaces}, which are examples of
log Calabi-Yau surfaces. 
\begin{defn}
\label{def-log-cy}(See \cite{Gross-Panda-Siebert}). Let $S$ be
a smooth surface and let $D\subset S$ a smooth divisor. Let furthermore
$\gamma\in\hhh_{2}(S,\Z)$ be nonzero. The pair $(S,D)$ is called
\emph{log Calabi-Yau with respect to} $\gamma$ if 
\begin{equation}
D\cdot\gamma=c_{1}(S)\cdot\gamma.\label{eq:log_cy}
\end{equation}
If equation \eqref{eq:log_cy} holds for all $\gamma\in\hhh_{2}(S,\Z)$,
which is the situation considered below, we abbreviate and say that
$(S,D)$ is a \emph{log Calabi-Yau surface pair}. Sometimes the divisor
$D$ is excluded from the notation.
\end{defn}
The discussion in \cite{Gross-Panda-Siebert} is concerned with any
log Calabi-Yau surface pair. For our purposes, we restrict to Del
Pezzo surfaces with a smooth anticanonical divisor. We mention in section
\ref{sub:Integrality-of-other-rel-bps} below an integrality result by Reineke in
\cite{reineke-refined}, which concerns relative BPS state counts associated to
blow ups of the projective plane relative to the toric divisor.%
\footnote{The precise geometry is more elaborate. In particular, in order to obtain
a smooth divisor, the singular points of the toric divisor are removed. The authors
in \cite{Gross-Panda-Siebert} prove that (contrary to expectation) invariants can be defined for this open geometry.%
}

As in the previous section, let $S$ be a Del Pezzo surface and denote
by $E$ a smooth effective anticanonical divisor on it. Then the
pair $(S,E)$ is log Calabi-Yau. This is the open Calabi-Yau geometry
considered in this section. This terminology is justified by the fact that the canonical bundle of
$S$ is trivial away from $E$. Let $\beta\in\hhh_{2}(S,\Z)$
be the class of a curve and set $w=E\cdot\beta$. Note that the moduli
stack of genus $0$ curves in $S$ is of virtual dimension $w-1$.
A generic curve representing $\beta$ would meet $E$ in $w$ points
of simple tangency. Instead, we can impose that the curve meets $E$
in fewer points with higher tangencies, cutting down the virtual dimension.
Considering the maximal case, denote by $\overline{M}(S/E,w)$ the
moduli stack which roughly parametrizes genus 0 relative stable maps $f:C\to S$ representing
$\beta$ and such that the image of $C$ meets $E$ in one point of
tangency $w$. Then $\overline{M}(S/E,w)$ is of virtual dimension
$0$ and the degree of its virtual fundamental class, 
\[
N_{S}[w]:=\int_{[\overline{M}(S/E,w)]^{vir}}1\in\mathbb{Q},
\]
is called the \emph{genus $0$ relative Gromov-Witten invariant of
degree $\beta$ and maximal tangency of $(S,E)$}. Denote by $\iota:P\to S$
a rigid element of $\overline{M}(S/E,w)$. For $k\geq1$, denote by
$M_{P}[k]$ the contribution of $k$-fold multiple covers of $P$
to $N_{S}[kw]$ (see \cite{Gross-Panda-Siebert} for precise definitions).
\begin{prop}
(Proposition $6.1$ in \cite{Gross-Panda-Siebert}).
\[
M_{P}[k]=\frac{1}{k^{2}}\binom{k(w-1)-1}{k-1}.
\]
\end{prop}
Consequently:
\begin{defn}
\label{def-rel-bps}(Paragraph §6.3 in \cite{Gross-Panda-Siebert}).
For $d\geq1$ , the relative BPS state counts $n_{S}[dw]\in\Q$ are
defined by means of the equality
\begin{equation}
\sum_{l=1}^{\infty}N_{S}[lw]\, q^{l}=\sum_{d=1}^{\infty}n_{S}[dw]\sum_{k=1}^{\infty}\frac{1}{k^{2}}\binom{k(dw-1)-1}{k-1}\, q^{dk}.\label{eq:defn_rel_bps_gen_fn}
\end{equation}
\end{defn}
\begin{conjecture}
\label{conjecture_GPS}(Conjecture $6.2$ in \cite{Gross-Panda-Siebert}).
Let $\beta\in\hhh_{2}(S,\Z)$ be an effective curve class and set
$w=\beta\cdot E$. Then, for all $d\geq1$,
\[
n_{S}[dw]\in\Z.
\]
\end{conjecture}

\subsection{Main result}

Our main result is based on the following theorem, which was proved
for $\ptwo$ by Gathmann in \cite{gathmannp2}. A proof for all Del
Pezzo surfaces was announced by Graber-Hassett.
\begin{thm}
\label{thm_GGH}(Gathmann for $\ptwo$ in \cite{gathmannp2}, for
general $S$ announced by Graber-Hassett). Let $S$ be a Del Pezzo
surface and denote by $E$ a smooth effective anticanonical divisor
on it. Let $\beta\in\hhh_{2}(S,\Z)$ be an effective curve class and
set $w=\beta\cdot E$. Then the following identity of Gromov-Witten
invariants holds:
\begin{equation}
N_{S}[w]=(-1)^{w+1}\, w\, I_{K_{S}}(\beta).\label{eq:GGH-formula2}
\end{equation}
\end{thm}
In the present paper, we prove the following theorem:
\begin{thm}
\label{thm:equivalence} Let $\beta\in\hhh_{2}(S,\Z)$ be an effective
non-zero primitive curve class. Consider two sequences of rational
numbers
$$
\left\{ N_{S}[dw]\right\} _{d\geq1} \text{ and } \left\{ I_{K_{S}}(d\beta)\right\} _{d\geq1},
$$
and assume that they are related, for all $d\geq1$, via (cf. equation \eqref{eq:GGH-formula2}):
\begin{equation}
N_{S}[dw]=(-1)^{dw+1}\, dw\, I_{K_{S}}(d\beta).\label{eq:GGH-formula}
\end{equation}
Define two sequences of rational numbers,
$$
\left\{ n_{S}[dw]\right\} _{d\geq1} \text{ and } \left\{ n_{d\beta}\right\} _{d\geq1},
$$
by means of the equations
\eqref{eq:defn_gv_gen_fn} and \eqref{eq:defn_rel_bps_gen_fn}. Then,
$$
n_{S}[dw]\in\Z \text{ for all } d\geq 1,
$$
if and only if
$$
dw\cdot n_{d\beta}\in\Z  \text{ for all } d\geq 1.
$$
\end{thm}
An immediate consequence is as follows:
\begin{cor}
\label{cor:implication}Conjecture $\ref{conjecture_GV}$ for $K_{S}$
implies conjecture $\ref{conjecture_GPS}$ for $(S,E)$.
\end{cor}
Per the integrality result of Peng in \cite{GV_inv_integral_local_toric_CY}
or of Konishi in \cite{Konishi_int_GV_inv} then:
\begin{cor}
\label{cor:integrality_of_toric_relative_BPS}Conjecture $\ref{conjecture_GPS}$
holds for toric Del Pezzo surfaces.
\end{cor}
Our proof of theorem \ref{thm:equivalence} employs some
of the same methods that Peng used in his proof of conjecture \ref{conjecture_GV}
for toric Del Pezzo surfaces, cf. \cite{GV_inv_integral_local_toric_CY}.
Namely, we rely on congruence relations between binomial coefficients.
In particular, lemma \ref{lem:binomial} below was stated in \cite{GV_inv_integral_local_toric_CY}.

\subsection{Relationship to Takahashi's work on log mirror symmetry}

It follows from theorem \ref{thm_GGH} that the relative BPS state
counts are related to the local BPS state counts (see lemma \ref{lem:formula_for_C}
below for the precise relationship). The local invariants are calculated
via mirror symmetry (see Chiang-Klemm-Yau-Zaslow in \cite{local_ms})
and thus, it is expected that the relative BPS state counts are directly
computed via mirror symmetry as well. It is not clear what the $B$-model
is though and there is at present no physical interpretation for these
relative Gromov-Witten invariants. A mirror symmetry conjecture in
this sense was formulated and explored by Takahashi in \cite{log_mirror_local}
for the projective plane. Let us note though that Takahashi considers
an alternative enumerative version of relative BPS state counts.

Takahashi develops logarithmic mirror symmetry for $\ptwo$ relative
to an elliptic curve $E$, and considers the following $A$-model invariants.
Let $d\geq1$. A degree $d$
curve in $\ptwo$ will meet $E$ in a $3d$-torsion point. Choose
a group structure on $E$ such that the zero element $0\in E$ is
a flex point. Choose $P\in E$ a point of order
$3d$ for the chosen group structure. Then $m_{d}$ is defined as
the number of rational degree $d$ curves in $\ptwo$ meeting $E$
only at $P$ in only one branch. The relative BPS state counts $n_{\ptwo}[3d]$
are a virtual extension of $m_{d}$ in the sense that the
rational curves virtually counted by $n_{\ptwo}[3d]$ are allowed to meet $E$
in any $3d$-torsion point, not just at $P$. The attribute virtual
is justified since the $n_{\ptwo}[3d]$ arise from Gromov-Witten invariants.
Based on his calculations and on the work by Gathmann in \cite{gathmannp2},
Takahashi conjectures that the $m_{d}$ are related to the local BPS
state counts $n_{d}$ of $\ptwo$ as follows.
\begin{conjecture}
\label{conj-Takahashi}(Takahashi in \cite{log_mirror_local})
\[
3d\, m_{d}=(-1)^{d+1}n_{d}.
\]
\end{conjecture}
The above conjecture provides an enumerative interpretation of the
invariants $n_{d}$. We prove a result analogous to conjecture \ref{conj-Takahashi}
in lemma \ref{lem:formula_for_C} below. Namely, that lemma provides
a linear relationship between the sets of invariants $n_{\ptwo}[3d]$
and $n_{d}$. It follows that the relative BPS state counts are calculated
from the periods of the mirror family. This is more generally true
for any Del Pezzo surface, since lemma \ref{lem:formula_for_C} holds
in that setting.

\subsection{\label{sub:Integrality-of-other-rel-bps}Integrality of relative
BPS state counts of another geometry}

The relative BPS state counts of definition \ref{def-rel-bps} are
defined in \cite{Gross-Panda-Siebert} in the more general setting of
log Calabi-Yau surface pairs (see definition \ref{def-log-cy} above),
and the integrality conjecture \ref{conjecture_GPS} is stated in that generality.
An example treated in detail throughout \cite{Gross-Panda-Siebert} is that of
the pairs consisting of blow ups of weighted projective planes and their toric divisors.\footnote{More precisely,
the singular points of the toric divisors are removed since the divisors of log Calabi-Yau surface pairs
are required to be smooth. The authors prove that the standard techniques yield
well-defined invariants.}
Consider $(a,b)\in\N^{2}$ determining
a weighted projective plane via the action 
\[
t\cdot(x,y.z)=(t^{a}x,t^{b}y,tz),
\]
where $t\in\IC^\times$. For $k\in\N$, let moreover
\[
\mathbf{P}=(\mathbf{P}_{a},\mathbf{P}_{b})
\]
consists of two ordered partitions
\[
\mathbf{P}_{a}=p_{1}+\cdots+p_{l_{a}} \text{ and } \,\mathbf{P}_{b}=p'_{1}+\cdots+p'_{l_{b}}
\]
of sizes $ak$, resp. $bk$. Then the authors consider relative Gromov-Witten
invariants, denoted by
\[
N_{a,b}[\mathbf{P}]\in\Q,
\]
which, roughly speaking, count genus $0$ maps with prescribed intersection
multiplicities along the toric divisors.%
\footnote{See \cite{Gross-Panda-Siebert} for the exact statements concerning
the geometry, the intersection multiplicities and the well-definedness
of the invariants.%
} Note that $k$ is implicit in the notation. In \cite{reineke-refined},
Reineke-Weist prove that the data of the Gromov-Witten invariants
$N_{a,b}[\mathbf{P}]$ is equivalent to the data given by the Euler
characteristic of moduli spaces of quiver representations. Using this
correspondence, the authors prove in \cite[corollary 11.4]{reineke-refined}
a variant of conjecture \ref{conjecture_GPS} above. They prove that BPS state
counts extracted from the projective space invariants $N_{1,1}[\mathbf{P}]$,
for appropriate partitions $\mathbf{P}$, are integers. The above corollary \ref{cor:integrality_of_toric_relative_BPS}
can be thought of an analogue of that result in the case of a smooth non-toric
anticanonical divisor.

\subsection{Relationship to the integrality of Donaldson-Thomas type invariants}

Gromov-Witten invariants are defined via intersection theory on the
stable map compactification.%
\footnote{Namely, $\overline{M}_{0,0}(S,\beta)$ compactifies the moduli space
of genus $0$ curves in $S$ of class $\beta$.%
} Donaldson-Thomas invariants are a closely related way of counting
curves, corresponding to the Hilbert space compactification.%
\footnote{See \cite{MNOPI,MNOPII}, where Maulik-Nekrasov-Okounkov-Pandharipande
conjecture the precise relationship.%
} A generalization
of Donaldson-Thomas invariants was proposed in \cite{Kontsevich-stab-struct}
by Kontsevich-Soibelman and in \cite{joyce-gen-dt} by Joyce-Song.
A special type of such invariants is considered by Reineke in \cite{reineke-degenerate}.
Fix $m\geq 1$. In \cite[definition 3.3]{reineke-degenerate}, the
author considers the following Donaldson-Thomas type invariants.
Let $n\geq 0$. Then $\dt_{n}^{(m)}$ encodes the Euler characteristic of some specific
non-commutative Hilbert schemes associated to the $m$-loop quiver.
More precisely, consider the free algebra on $m$ generators
$F^{(m)}:=\IC\langle x_1,\dots,x_m\rangle$. Consider moreover the
non-commutative Hilbert scheme $\hilb_n^{(m)}$ parametrizing left ideals $I$ in $F^{(m)}$
of codimension $n$, i.e.\ such that $\dim_{\IC} F^{(m)} / I = n$. Consider furthermore the generating
function of Euler characteristics
$$
F(t):=\sum_{n\geq 0} \chi\left( \hilb_n^{(m)} \right) t^n \in \Z[[t]].
$$
Then the Donaldson-Thomas type invariants $\dt_{n}^{(m)}$
are defined as rational numbers via the the equality of power series
$$
F\left((-1)^{m-1} \; t\right) = \prod_{n\geq 1} \left( 1-t^n \right)^{-(-1)^{(m-1)n} \; n \dt_n^{(m)}}.
$$
The shape of the above definition is motivated by \cite{Kontsevich-stab-struct}.
Then, \cite[theorem 3.2]{reineke-degenerate} states that $\dt_{n}^{(m)}\in\N$
and provides a formula in terms of sums of binomial
coefficients. This result is proven for a more general class of examples
by Reineke in \cite[theorem 5.9 and formula (15)]{reineke-cohomology}. It is absolutely remarkable that the invariants
$\dt_{n}^{(m)}$ are exactly the coefficients of the transformation matrix $C$
of lemma \ref{lem:formula_for_C} below:
\begin{equation}
\label{explicit}
C_{st} = \dt_{s/t}^{(tw-1)}.
\end{equation}
Here $s,t\geq 1$ are such that $s/t\in\N$, and $w$ is as in section \ref{intro-rel-BPS}.
The matrix $C$ transforms the relative BPS state counts $n_S[dw]$ from section \ref{intro-rel-BPS}
into the local BPS state counts $n_{d\beta}$ from section \ref{intro-loc-BPS}.
Albeit formula \eqref{explicit} provides a direct numerical connection, we are as of now not
aware of a geometric connection as to why the transformation matrix from relative to local
BPS state counts should be given by the invariants $\dt_{s/t}^{(tw-1)}$. Our proof
of the integrality of the coefficients $C_{st}$, via analysis of
congruence relations of binomial coefficients, has much in common with the proof of
\cite[theorem 3.2]{reineke-degenerate}. That theorem however
is proven in the more general setting of \cite{reineke-cohomology}.
Avoiding the extra formalism of \cite{reineke-cohomology}, our proof is more direct.

\subsection{Outline}

The proof of theorem \ref{thm:equivalence} is split into two parts.
In section \ref{sec:Combinatorics}, lemma \ref{lem:formula_for_C}
states the precise relationship between the local and relative BPS
state counts that we consider. Each set of invariants is related to
the other by means of an invertible matrix. In section \ref{sec:Integrality},
we analyze congruence classes relating to the entries of this matrix.
We prove that each entry is integer valued, which proves theorem \ref{thm:equivalence}.
\begin{acknowledgement*}
The first author would like to thank T. Graber for introducing him
to the subject of Gromov-Witten theory and for many enlightening discussions
on curve counting that formed the basis for the present paper. The
first author would like to extend special thanks to Y. Ruan, who has
provided valuable guidance on the aspects of this paper relating to
mirror symmetry. The authors would like to thank N. Yui, R. Abouaf
and M. Florence for many helpful comments that improved the quality
and readability of the paper. The authors would like to extend special
thanks to the referees, whose reports greatly improved the quality
of the present paper. Part of the research reported here was
performed while the authors were students at the California Institute
of Technology. This paper was completed while the first author was
in residence at the Fields Institute for the thematic program: Calabi-Yau
Varieties: Arithmetic, Geometry and Physics, July-December 2013. The
first author would like to thank the Fields Institute for its hospitality
and generous support.
\end{acknowledgement*}

\section{Combinatorics\label{sec:Combinatorics}}

Let $S$ be a Del Pezzo surface with smooth effective anticanonical
divisor $E$ and let $\beta\in\hhh_{2}(S,\Z)$ be a non-zero effective
primitive curve class. Consider two sequences of rational numbers
\[
\left\{ N_{S}[dw]\right\} _{d\geq1}\text{ and } \, \left\{ I_{K_{S}}(d\beta)\right\} _{d\geq1},
\]
which we assume to be related, for all $d\geq1$, by equation \eqref{eq:GGH-formula}.
Define two sequences of rational numbers, 
\[
\left\{ n_{S}[dw]\right\} _{d\geq1}\text{ and } \, \left\{ n_{d\beta}\right\} _{d\geq1},
\]
by means of the equations \eqref{eq:defn_gv_gen_fn} and \eqref{eq:defn_rel_bps_gen_fn}.
Note that formula \eqref{eq:defn_rel_bps_gen_fn} is equivalent to
the set of equations:
\begin{equation}
N_{S}[dw]=\sum_{k|d}\frac{1}{k^{2}}\binom{k(\frac{d}{k}w-1)-1}{k-1}n_{S}[dw/k].\label{eq:defn-rel}
\end{equation}
Combining the formulas \eqref{eq:def_gv_form}, \eqref{eq:GGH-formula}
and \eqref{eq:defn-rel} yields the following collection of formulas:
\begin{equation}
\sum_{k|d}\frac{1}{k^{2}}\binom{k(\frac{d}{k}w-1)-1}{k-1}n_{S}[dw/k]=(-1)^{dw+1}\, dw\sum_{k|d}\frac{1}{k^{3}}\, n_{d\beta/k}.\label{eq:rel_sum}
\end{equation}

Fix a positive integer $N$ and, for $1\leq d\leq N$, consider the
formulas \eqref{eq:rel_sum}. In matrix form, this collection of formulas
is expressed as 
\begin{equation}
R\,\left[n_{S}[dw]\right]_{d}=A\cdot L\cdot A^{-1}\,\left[(-1)^{dw+1}\, dw\, n_{d\beta}\right]_{d},\label{eq:matrices}
\end{equation}
where $R$, $A$ and $L$ are the following lower triangular $N\times N$
matrices:
\begin{align*}
R_{ij} & :=\begin{cases}
\frac{1}{(i/j)^{2}}\binom{i/j\,(jw-1)-1}{i/j-1} & \text{if }j|i,\\
0 & \text{else;}
\end{cases}\\
A_{ij} & :=(-1)^{iw+1}\, iw\,\delta_{ij};\\
L_{ij} & :=\begin{cases}
\frac{1}{(i/j)^{3}} & \text{if }j|i,\\
0 & \text{else.}
\end{cases}
\end{align*}
Note that $R$ is lower triangular and has determinant $1$.
\begin{notation*}
For an integer $n$, denote by $\omega(n)$ the number of primes (not
counting multiplicities) in the prime factorization of $n$. Moreover,
let
\[
I(n):=\left\{ k\in\N\,:\, k|n\text{ and }n/k\text{ is square-free}\right\} .
\]
\end{notation*}
\begin{lem}
\label{lem:formula_for_C}Define the $N\times N$ matrix $C$ as follows.
If $t|s$, set
\begin{equation}
C_{st}:=\frac{(-1)^{sw}}{(s/t)^{2}}\sum_{k\in I(s/t)}(-1)^{\omega\left(s/kt\right)}(-1)^{ktw}\binom{k(tw-1)-1}{k-1}.\label{eq:C_st}
\end{equation}
If $t\nmid s$, set $C_{st}=0$. Then the sequences
$$
\left\{ n_{S}[dw]\right\} \text{ and } \left\{ (-1)^{dw+1}\, dw\, n_{d\beta}\right\},
$$
for $1\leq d\leq N$, are related via 
\begin{equation}
C\cdot\left[n_{S}[dw]\right]_{d}=\left[(-1)^{dw+1}\, dw\, n_{d\beta}\right]_{d}.\label{eq:relation_local_rel_bps}
\end{equation}
Moreover, $C$ has determinant $1$ and is lower triangular. Thus,
by Cramer's rule,
\[
C\text{ integral }\Longleftrightarrow C^{-1}\text{ integral.}
\]
\end{lem}
\begin{proof}
We start by writing $L=B\cdot\tilde{L}\cdot B^{-1}$, where
\begin{align*}
\tilde{L}_{ij} & =\begin{cases}
1 & \text{if }j|i,\\
0 & \text{else};
\end{cases}\\
B_{ij} & =\frac{1}{i^{3}}\,\delta_{ij}.
\end{align*}
By M\"obius inversion the inverse of $\tilde{L}$ is given by
\[
\left(\tilde{L}^{-1}\right)_{ij}=\begin{cases}
(-1)^{\omega(i/j)} & \text{if }j|i\text{ and }i/j\text{ is square-free},\\
0 & \text{else.}
\end{cases}
\]
Then,
\[
\left(AB\right)_{ij}=(-1)^{iw+1}\,\frac{w}{i^{2}}\,\delta_{ij},
\]
and
\[
\left((AB)^{-1}\right)_{ij}=(-1)^{iw+1}\,\frac{i^{2}}{w}\,\delta_{ij}.
\]
It follows from formula \eqref{eq:matrices} that a matrix $C$ satisfying \eqref{eq:relation_local_rel_bps} is given by
\[
C=A L^{-1} A^{-1}\cdot R=AB\cdot\tilde{L}^{-1}\cdot(AB)^{-1}\cdot R.
\]
The matrix $C$ is lower triangular and has determinant 1, as this is the case for
both $A L^{-1} A^{-1}$ and $R$. A calculation then yields
\[
\left(AB\cdot\tilde{L}^{-1}\right)_{sr}=\begin{cases}
(-1)^{sw+1}\,\frac{w}{s^{2}}\,(-1)^{\omega(s/r)} & \text{if }r|s\text{ and }s/r\text{ is square-free},\\
0 & \text{else;}
\end{cases}
\]
and
\[
\left((AB)^{-1}\cdot R\right)_{rt}=\begin{cases}
(-1)^{rw+1}\,\frac{r^{2}}{w}\,\frac{1}{(r/t)^{2}}\binom{r/t\,(tw-1)-1}{r/t-1} & \text{if }t|r,\\
0 & \text{else.}
\end{cases}
\]
If $t$ does not divide $s$, then there is no integer $r$ such that
$t|r|s$, so that $C_{st}=0$. If, however, $t|s$, then
\begin{align*}
C_{st} & =\frac{(-1)^{sw+1}}{(s/t)^{2}}\,\sum(-1)^{\omega(s/r)}\,(-1)^{rw+1}\,\binom{r/t\,(tw-1)-1}{r/t-1},
\end{align*}
where the sum runs over all $r$ such that $t|r|s$ and such that
$s/r$ is square-free. Set $k=r/t$, so that, for $t$ dividing $s$,
\begin{align*}
C_{st} & =\frac{(-1)^{sw}}{(s/t)^{2}}\,\sum_{k\in I(s/t)}(-1)^{\omega(s/kt)}\,(-1)^{ktw}\,\binom{k\,(tw-1)-1}{k-1},
\end{align*}
finishing the proof.
\end{proof}
Lemma \ref{lem:formula_for_C} reduces theorem \ref{thm:equivalence}
to proving that the coefficients of the matrix $C$ are integers.
This is achieved in the lemmas \ref{lem:integrality_most_cases} and
\ref{lem:integrality_special_case} of the next section.

\section{Integrality\label{sec:Integrality}}

We start by stating the following lemma, which follows directly from
the proof of lemma $A.1$ of \cite{GV_inv_integral_local_toric_CY}.
\begin{lem}
\label{lem:binomial}(Peng) Let $a,b$ and $\alpha$ be positive integers
and denote by $p$ a prime number. If $p=2$, assume furthermore that
$\alpha\geq2$. Then
\[
\binom{p^{\alpha}a-1}{p^{\alpha}b-1}\equiv\binom{p^{\alpha-1}a-1}{p^{\alpha-1}b-1}\mod\left(p^{2\alpha}\right).
\]
\end{lem}
In the case that $p=2$ and $\alpha=1$, we have the following lemma:
\begin{lem}
\label{lem:binomial_special_case}Let $k\geq1$ be odd and let $a$
be a positive integer. Then
\[
\binom{2ka-1}{2k-1}\equiv(-1)^{a+1}\binom{ka-1}{k-1}\mod\left(4\right).
\]
\end{lem}
\begin{proof}
Note that
\begin{align*}
\binom{2ka-1}{2k-1} & =\frac{2ka-1}{2k-1}\cdot\frac{2ka-2}{2k-2}\cdots\frac{2ka-2k+2}{2}\cdot\frac{2ka-2k+1}{1}\\
 & =\frac{2ka-1}{2k-1}\cdot\frac{ka-1}{k-1}\cdots\frac{ka-k+1}{1}\cdot\frac{2ka-2k+1}{1}\\
 & =\frac{(ka-1)(ka-2)\cdots(ka-k+1)}{(k-1)(k-2)\cdots1}\\%\cdot\frac{(2ka-1)(2ka-3)\cdots(2ka-2k+1)}{(2k-1)(2k-3)\cdots1}\\
 %& =\frac{(ka-1)(ka-2)\cdots(ka-k+1)}{(k-1)(k-2)\cdots1}
 & \hspace{2.5mm} \cdot\frac{(2ka-1)(2ka-3)\cdots(2ka-2k+1)}{(2k-1)(2k-3)\cdots1}\\
 & =\binom{ka-1}{k-1}\cdot\frac{(2ka-1)(2ka-3)\cdots(2ka-2k+1)}{(2k-1)(2k-3)\cdots1},
\end{align*}
and hence
\begin{align*}
 & \binom{2ka-1}{2k-1}+(-1)^{a}\binom{ka-1}{k-1}\\
= & \binom{ka-1}{k-1}\left((-1)^{a}+\frac{(2ka-1)(2ka-3)\cdots(2ka-2k+1)}{(2k-1)(2k-3)\cdots1}\right).
\end{align*}
It thus suffices to show that
\begin{equation}
\frac{(2ka-1)(2ka-3)\cdots(2ka-2k+1)}{(2k-1)(2k-3)\cdots1}\equiv(-1)^{a+1}\mod\left(4\right).\label{eq:long_expression}
\end{equation}
Suppose first that $a$ is even, so that the left-hand-side of \eqref{eq:long_expression}
is congruent to
\begin{align*}
\equiv & \, \frac{(-1)(-3)\cdots(-(2k-3))(-(2k-1))}{1\cdot3\cdots(2k-3)(2k-1)}\\
\equiv & \, (-1)^{k}\equiv \, (-1)^{a+1}\mod(4),
\end{align*}
where the last congruence follows form the fact that $k$ is odd.
Suppose now that $a$ is odd. Then the left-hand-side of the expression
\eqref{eq:long_expression} is congruent to
\begin{align*}
\equiv & \, \frac{2k(a-1)+(2k-1)}{2k-1}\cdot\frac{2k(a-1)+(2k-3)}{2k-3}\cdots\frac{2k(a-1)+1}{1}\\
\equiv & \, 1\equiv \, (-1)^{a+1}\mod\left(4\right).
\end{align*}
\end{proof}
We return to the proof of theorem \ref{thm:equivalence}. If $s=t$,
then $C_{st}=1$ is an integer, and we may thus henceforth assume that $t|s$,
but $t\neq s$. Let $p$ be a prime number and $\alpha$ a positive
integer. For an integer $n$, we use the notation
\[
p^{\alpha}||n,
\]
to mean that $p^{\alpha}|n$, but $p^{\alpha+1}\nmid n$. In order
to prove that $C_{st}\in\Z$, we show the following: If $p$ is a
prime number such that
\[
p^{\alpha}||\frac{s}{t},
\]
then
\begin{equation}
p^{2\alpha}|\sum_{k\in I(s/t)}(-1)^{\omega(s/kt)}\,(-1)^{ktw}\,\binom{k\,(tw-1)-1}{k-1}.\label{eq:divisibility}
\end{equation}
Fix a prime number $p$ and a positive integer $\alpha$ such that
\[
p^{\alpha}||\frac{s}{t}.
\]
We regroup the sum over $k\in I(s/t)$ as follows. Let $k\in I(s/t)$.
For $s/kt$ to be square-free, it is necessary that $p^{\alpha-1}|k$.
This splits into the two cases
\[
p^{\alpha-1}||k,
\]
in which case $k/p^{\alpha-1}\in I(s/p^{\alpha}t)$; and
\[
p^{\alpha}||k,
\]
so that $k/p^{\alpha-1}=p\, l$ for $l\in I(s/p^{\alpha}t)$. Regrouping
the terms of the sum \eqref{eq:divisibility} accordingly yields
\[
\sum_{l\in I(s/p^{\alpha}t)}\sum_{k\in\left\{ p^{\alpha-1}l,p^{\alpha}l\right\} }(-1)^{\omega(s/kt)}\,(-1)^{ktw}\,\binom{k\,(tw-1)-1}{k-1}.
\]
Thus, it suffices to show that for all $l\in I(s/p^{\alpha}t)$,
\[
f(l):=\sum_{k\in\left\{ p^{\alpha-1}l,p^{\alpha}l\right\} }(-1)^{\omega(s/kt)}\,(-1)^{ktw}\,\binom{k\,(tw-1)-1}{k-1}\equiv0\mod\left(p^{2\alpha}\right),
\]
which we proceed to prove. There are two cases. In the above sum, either the sign $(-1)^{\omega(s/kt)}\,(-1)^{ktw}$
changes or it does not. The only case where the sign does
not change is when $p=2$, $\alpha=1$, and both $t$ and $w$ are
odd.
\begin{lem}
\label{lem:integrality_most_cases}Assume that either $p\neq2$ or,
if $p=2$, that $\alpha>1$. Then
\[
f(l)\equiv0\mod\left(p^{2\alpha}\right).
\]
\end{lem}
\begin{proof}
In this case,
\begin{align*}
f(l) & =\pm\left(\binom{p^{\alpha}l(tw-1)-1}{p^{\alpha}l-1}-\binom{p^{\alpha-1}l(tw-1)-1}{p^{\alpha-1}l-1}\right)\\
 & \equiv0\mod\left(p^{2\alpha}\right),
\end{align*}
by lemma \ref{lem:binomial}.\end{proof}
\begin{lem}
\label{lem:integrality_special_case}Assume that $p=2$ and that $\alpha=1$.
Then
\[
f(l)\equiv0\mod\left(4\right).
\]
\end{lem}
\begin{proof}
In this case,
\begin{align*}
f(l) & =\pm\left(\binom{2l(tw-1)-1}{2l-1}+(-1)^{tw-1}\binom{l(tw-1)-1}{l-1}\right)\\
 & \equiv0\mod\left(4\right),
\end{align*}
follows from lemma \ref{lem:binomial_special_case}.
\end{proof}
Therefore in both cases $f(l)$ is divisible by $p^{2\alpha}$, and this
finishes the proof that the entries of $C$ are integers. Consequently,
by lemma \ref{lem:formula_for_C}, the proof of theorem \ref{thm:equivalence}
is complete.
\bibliographystyle{amsalpha}
\providecommand{\bysame}{\leavevmode\hbox to3em{\hrulefill}\thinspace}
\providecommand{\MR}{\relax\ifhmode\unskip\space\fi MR }
\providecommand{\MRhref}[2]{%
  \href{http://www.ams.org/mathscinet-getitem?mr=#1}{#2}
}
\providecommand{\href}[2]{#2}

\end{document}